\documentclass[12pt]{amsart}
\usepackage{amsmath,amsthm,amsfonts,amssymb,mathrsfs}
\date{\today}

\usepackage[cp1251]{inputenc}         

\usepackage[ukrainian]{babel} 
\usepackage{amsmath,amsthm,amsfonts,amssymb,mathrsfs}
\usepackage{eucal}

\input xymatrix
\xyoption{all}

\usepackage{color}

\usepackage{amssymb,cite}

\usepackage{hyperref}

\usepackage{tikz}

\newtheorem{theorem}{Теорема}

\newtheorem{proposition}{Твердження}
\newtheorem{corollary}{Наслiдок}
\newtheorem{lemma}{Лема}
\theoremstyle{definition}
\newtheorem{remark}{Зауваження}

\begin{document}

\title[Напiвгрупа скiнченних часткових порядкових iзоморфiзмiв обмеженого рангу]{Напiвгрупа скiнченних часткових порядкових iзоморфiзмiв обмеженого рангу нескiнченно\"{\i} лiнiйно впорядковано\"{\i}  множини}

\author[Олег~Гутік, Максим Щипель]{Олег~Гутік, Максим Щипель}
\address{Львівський національний університет ім. Івана Франка, Університецька 1, Львів, 79000, Україна}
\email{oleg.gutik@lnu.edu.ua, maksym.shchypel@lnu.edu.ua}



\begin{abstract}
Ми вивчаємо алгебричні властивості напівгрупи $\mathscr{O\!\!I\!}_n(L)$  скiнченних часткових порядкових iзоморфiзмiв рангу $\leq n$ нескінченної лiнiйно впорядкованої  множини $(L,\leqslant)$. Зокрема описано її ідемпотенти, природний частковий порядок та відношення Ґріна на $\mathscr{O\!\!I\!}_n(L)$. Доведено, що напівгрупа $\mathscr{O\!\!I\!}_n(L)$ стійка та містить щільний ряд ідеалів, а також, що всі конґруенції на напівгрупі $\mathscr{O\!\!I\!}_n(L)$ є конґруенціями Ріса.

\bigskip
\noindent
\emph{Oleg Gutik, Maksym Shchypel, \textbf{The semigroup of finite partial order isomorphisms of a bounded rank of an infinite linear ordered set}.}

\smallskip
\noindent
We study algebraic properties of the semigroup $\mathscr{O\!\!I\!}_n(L)$ of finite partial order isomorphisms of the rank $\leq n$ of an infinite linearly ordered set $(L,\leqslant)$. In particular we describe its idempotents, the natural partial order and Green's relations on $\mathscr{O\!\!I\!}_n(L)$. It is proved that the semigroup $\mathscr{O\!\!I\!}_n(L)$ is stable and it contains tight ideal series. Moreover, we show that the semigroup  $\mathscr{O\!\!I\!}_n(L)$ admits only Rees' congruences and every its homomorphic image is a semigroup with tight ideal series.

\end{abstract}

\maketitle



У цій праці ми користуємося термінологією з монографій \cite{Clifford-Preston-1961, Clifford-Preston-1967, Lawson-1998, Petrich-1984}.

Якщо визначене часткове відображення $\alpha\colon X\rightharpoonup Y$ з множини $X$ у множину $Y$, то через $\operatorname{dom}\alpha$ i $\operatorname{ran}\alpha$ будемо позначати його \emph{область визначення} та \emph{область значень}, відповідно, а через $(x)\alpha$ і $(A)\alpha$ --- образи елемента $x\in\operatorname{dom}\alpha$ та підмножини $A\subseteq\operatorname{dom}\alpha$ при частковому відображенні $\alpha$, відповідно.
Також через $\operatorname{rank}\alpha$ будемо позначати ранг часткового відображення $\alpha\colon X\rightharpoonup Y$, тобто $\operatorname{rank}\alpha=|\operatorname{dom}\alpha|$.

Якщо $S$ --- напівгрупа, то визначатимемо відношення Ґріна $\mathscr{R}$, $\mathscr{L}$, $\mathscr{D}$, $\mathscr{H}$ і $\mathscr{J}$ на $S$ так:
\begin{align*}
    &\qquad a\mathscr{R}b \mbox{ тоді і лише тоді, коли } aS^1=bS^1;\\
    &\qquad a\mathscr{L}b \mbox{ тоді і лише тоді, коли } S^1a=S^1b;\\
    &\qquad a\mathscr{J}b \mbox{ тоді і лише тоді, коли } S^1aS^1=S^1bS^1;\\
    &\qquad \mathscr{D}=\mathscr{L}\circ\mathscr{R}=          \mathscr{R}\circ\mathscr{L};\\
    &\qquad \mathscr{H}=\mathscr{L}\cap\mathscr{R}.
\end{align*}
 (див. означення в \cite[\S2.1]{Clifford-Preston-1961} або \cite{Green-1951}).

Відношення еквівалентності $\mathfrak{K}$ на напівгрупі $S$ називається \emph{конґруенцією}, якщо для елементів $a$ i $b$ напівгрупи $S$ з того, що виконується умова $(a,b)\in\mathfrak{K}$ випливає, що $(ca,cb), (ad,bd) \in\mathfrak{K}$, для всіх $c,d\in S$. Відношення $(a,b)\in\mathfrak{K}$ також будемо записувати $a\mathfrak{K}b$, і в цьому випадку будемо говорити, що \emph{елементи $a$ i $b$ є $\mathfrak{K}$-еквівалентними}. На кожній напівгрупі $S$ існують наступні конґруенції: \emph{універсальна} $\mathfrak{U}_S=S\times S$ та \emph{одинична} (\emph{діагональ}) $\Delta_S=\{(s,s)\colon s\in S\}$. Такі конґруенції називаються \emph{тривіальними}. Кожен двобічний ідеал $I$ напівгрупи $S$ породжує на ній конґруенцію Ріса: $\mathfrak{K}_I=(I\times I)\cup\Delta_S$.

Надалі через $E(S)$ позначатимемо множину ідемпотентів напівгрупи $S$. Напівгрупа ідемпотентів називається \emph{в'язкою}, а комутативна напівгрупа ідемпотентів --- \emph{напівґраткою}.

Якщо $S$~--- напівгрупа, то на $E(S)$ визначено частковий порядок:
$
e\preccurlyeq f
$   тоді і лише тоді, коли
$ef=fe=e$.
Так означений частковий порядок на $E(S)$ називається \emph{при\-род\-ним}.

Напівгрупа $S$ називається \emph{інверсною}, якщо для довільного елемента $s\in S$ існує єдиний елемент $s^{-1}\in S$ такий, що $ss^{-1}s=s$ i $s^{-1}ss^{-1}=s^{-1}$~\cite{Wagner-1952}. В інверсній напівгрупі $S$ вище означений елемент $s^{-1}$ називається \emph{інверсним до} $s$.

Означимо відношення $\preccurlyeq$ на інверсній напівгрупі $S$ так:
$
    s\preccurlyeq t
$
тоді і лише тоді, коли $s=te$, для деякого ідемпотента $e\in S$. Так означений частковий порядок назива\-єть\-ся \emph{при\-род\-ним част\-ковим порядком} на інверсній напівгрупі $S$~\cite{Wagner-1952}. Очевидно, що звуження природного часткового порядку $\preccurlyeq$ на інверсній напівгрупі $S$ на її в'язку $E(S)$ є при\-род\-ним частковим порядком на $E(S)$.

Через $\mathscr{I}(X)$ позначимо множину всіх часткових взаємнооднозначних перетворень множини $X$ разом з такою напівгруповою операцією
\begin{equation*}
    x(\alpha\beta)=(x\alpha)\beta \quad \mbox{якщо} \quad
    x\in\operatorname{dom}(\alpha\beta)=\{
    y\in\operatorname{dom}\alpha\colon
    y\alpha\in\operatorname{dom}\beta\}, \qquad \mbox{для} \quad
    \alpha,\beta\in\mathscr{I}_\lambda.
\end{equation*}
Напівгрупа $\mathscr{I}(X)$ називається  \emph{симетричною інверсною напівгрупою} або \emph{симетричним інверсним моноїдом} над множиною $X$~(див.  \cite{Clifford-Preston-1961}). Симетрична інверсна напівгрупа введена В. В. Вагнером у працях~\cite{Wagner-1952a, Wagner-1952} і вона відіграє важливу роль у теорії напівгруп. Надалі, якщо для $\alpha,\beta\in \mathscr{I}(X)$ виконуються умови $\operatorname{dom}\alpha\subseteq \operatorname{dom}\beta$ i $(x)\beta=(x)\alpha$ для довільного $x\in \operatorname{dom}\alpha$. то будемо писати $\alpha\subseteq \beta$.

Надалі будемо вважати, що $(L,\leqslant)$~--- нескінченна лінійно впорядкована множина. Для елементів $x,y\in L$ умову $x\leqslant y$ i $x\neq y$ записуватимемо так: $x<y$. Елемент $\alpha\in\mathscr{I}(L)$ називається \emph{частковим порядковим iзоморфiзмом}, якщо  для довільних $x_1,x_2\in \operatorname{dom}\alpha$ з $x_1\leqslant x_2$ випливає $(x_1)\alpha\leqslant (x_2)\alpha$. Позаяк $\leqslant$~--- лінійний порядок на $L$, то для $\alpha\in\mathscr{I}(L)$ і для до\-віль\-них $x_1,x_2\in \operatorname{dom}\alpha$ з $(x_1)\alpha\leqslant (x_2)\alpha$ випливає $x_1\leqslant x_2$. Очевидно, що композиція часткових порядкових iзоморфiзмiв лінійно впорядкованої множини $(L,\leqslant)$ є частковим порядковим iзоморфiзмом, і обернене часткове відображення до часткового порядкового iзоморфiзму є знову частковим порядковим iзоморфiзмом. Через $\mathscr{O\!\!I\!}(L)$ позначимо напівгрупу всіх часткових порядкових iзоморфiзмів лінійно впорядкованої множини $(L,\leqslant)$. Очевидно, що $\mathscr{O\!\!I\!}(L)$~--- інверсна напівгрупа симетричного інверсного моноїда $\mathscr{I}(L)$ над множиною $L$.

Для довільного натурального числа $n$ визначимо
\begin{equation*}
  \mathscr{O\!\!I\!}_n(L)=\left\{\alpha\in\mathscr{O\!\!I\!}(L)\colon |\operatorname{dom}\alpha|\leq n\right\}.
\end{equation*}
Очевидно, що $\mathscr{O\!\!I\!}_n(L)$~--- інверсна піднапівгрупа симетричного інверсного моноїда $\mathscr{O\!\!I\!}(L)$ над множиною $L$. Надалі напівгрупу $\mathscr{O\!\!I}\!_n(L)$ будемо називати \emph{напiвгрупою скiнченних часткових порядкових iзоморфiзмiв рангу $\leq n$ лiнiйно впорядкованої  множини} $(L,\leqslant)$. Надалі якщо елемент $\alpha$ напівгрупи $\mathscr{O\!\!I}\!_n(L)$ відображає $x_1$ у $y_1$, $\ldots$, $x_k$ у $y_k$, де $x_1,\ldots x_k,y_1,\ldots y_k\in L$, $x_1<\cdots< x_k$, $y_1<\cdots< y_k$, $k\leq n$, то записуватимемо так;
\begin{equation*}
  \alpha=
\begin{pmatrix}
x_1 & \cdots & x_k\\
y_1 & \cdots & y_k
\end{pmatrix}.
\end{equation*}
Порожнє часткове перетворення множини $L$ будемо позначати символом $\boldsymbol{0}$. Очевидно, що $\boldsymbol{0}$ --- нуль напівгрупи $\mathscr{O\!\!I}\!_n(L)$.

Однією з класичних задач теорії напівгруп перетворень є дослідження будови напівгрупи перетворень множини, які зберігають структуру множини (геометрію, частковий порядок, топологію), зокрема, коли ці перетворення є локальними, тобто частковими еквівалентностями (частковими ізометріями, частковими  порядковими ізоморфізмами, частковими гомеоморфізмами, тощо) \cite{Gluskin-Schein-Shneperman-Yaroker-1967, Magill-1976}.
У цій праці ми досліджуємо алгебричні властивості напівгрупи $\mathscr{O\!\!I\!}_n(L)$ скінченних часткових порядкових iзоморфiзмiв обмеженого рангу лінійно впорядкованої множини $(L,\leqslant)$. Зокрема описано її ідемпотенти, природний частковий порядок та відношення Ґріна на $\mathscr{O\!\!I\!}_n(L)$. Доведено, що напівгрупа $\mathscr{O\!\!I\!}_n(L)$ стійка та містить щільний ряд ідеалів, а також, що всі конґруенції на напівгрупі $\mathscr{O\!\!I\!}_n(L)$ є конґруенціями Ріса. 


\begin{lemma}\label{lemma-2.1}
Нехай $(L,\leqslant)$~--- нескінченна лінійно впорядкована множина. Тоді для довільного натурального числа $k$ і для довільних $x_1,\ldots x_k,y_1,\ldots y_k\in L$ таких, що $x_1<\cdots< x_k$, $y_1<\cdots< y_k$ існує єдиний частковий порядковий ізоморфізм $\alpha\colon L\rightharpoonup L$ такий, що $\operatorname{dom}\alpha=\left\{x_1,\ldots x_k\right\}$, $\operatorname{ran}\alpha=\left\{y_1,\ldots y_k\right\}$. Більше того, частковий порядковий ізоморфізм $\alpha\colon L\rightharpoonup L$ визначається так: $(x_1)\alpha=y_1, \ldots,(x_k)\alpha=y_k$.
\end{lemma}

\begin{proof}
Позаяк $(L,\leqslant)$~--- лінійно впорядкована множина, то кожна скінченна підмножина в $(L,\leqslant)$ має найменший елемент. За індукцією визначаємо  частковий порядковий ізоморфізм $\alpha\colon L\rightharpoonup L$ так: $(x_1)\alpha=y_1, \ldots,(x_k)\alpha=y_k$. Позаяк множини $\left\{x_1,\ldots x_k\right\}$, $\left\{y_1,\ldots y_k\right\}$ скінченні, то такий частковий порядковий ізоморфізм $\alpha$ єдиний.
\end{proof}

Нагадаємо \cite{Koch-Wallace-1957}, що напівгрупа $S$ називається \emph{стійкою}, якщо
\begin{enumerate}
  \item[(1)] $a,b\in S$ і з $Sa\subseteq Sab$ випливає, що $Sa=Sab$,
  \item[(2)] $a,b\in S$ і з $aS\subseteq baS$ випливає, що $aS=baS$.
\end{enumerate}

\begin{theorem}\label{theorem-2.1}
Нехай $(L,\leqslant)$~--- нескінченна лінійно впорядкована множина. Тоді для довільного натурального числа $n$ напівгрупа $\mathscr{O\!\!I}\!_n(L)$ є стійкою.
\end{theorem}

\begin{proof}
Припустимо, що для деяких елементів $\alpha,\beta\in \mathscr{O\!\!I}\!_n(L)$  виконується вклю\-чен\-ня
\begin{equation}\label{eq-2.1}
\alpha\mathscr{O\!\!I}\!_n(L)\subseteq \beta\alpha\mathscr{O\!\!I}\!_n(L).
\end{equation}
Позаяк $\mathscr{O\!\!I}\!_n(L)$ --- інверсна напівгрупа, то за теоремою 1.17 з \cite{Clifford-Preston-1961} маємо, що
\begin{equation*}
\alpha\mathscr{O\!\!I}\!_n(L)=\alpha\alpha^{-1}\mathscr{O\!\!I}\!_n(L) \quad \hbox{i} \quad \beta\alpha\mathscr{O\!\!I}\!_n(L)=\beta\alpha\alpha^{-1}\beta^{-1}\mathscr{O\!\!I}\!_n(L),
\end{equation*}
а отже,
\begin{equation}\label{eq-2.2}
\alpha\alpha^{-1}\mathscr{O\!\!I}\!_n(L)\subseteq \beta\alpha\alpha^{-1}\beta^{-1}\mathscr{O\!\!I}\!_n(L).
\end{equation}
З включення \eqref{eq-2.2} випливає, що $\operatorname{dom}\alpha\subseteq \operatorname{dom}\beta$ i $\operatorname{dom}\alpha\subseteq \operatorname{ran}\beta$. З включення \eqref{eq-2.1} маємо, що існує такий елемент $\gamma\in\mathscr{O\!\!I}\!_n(L)$, що
\begin{equation}\label{eq-2.3}
\alpha=\beta\alpha\gamma.
\end{equation}

Нехай $\operatorname{rank}\alpha=k\leq n$ i $\operatorname{dom}\alpha=\{x_1,\ldots,x_k\}$, причому $x_1<\cdots< x_k$ в $(L,\leqslant)$. З рівності \eqref{eq-2.3} випливає, що $(x_1)\beta^{-1}=x_1$, $\ldots,$ $(x_k)\beta^{-1}=x_k$, оскільки $(L,\leqslant)$ --- лінійно впорядкована множина, $\alpha,\beta\in\mathscr{O\!\!I}\!_n(L)$ і $x_1<\cdots< x_k$ в $(L,\leqslant)$. Отже, отримуємо, що $(x_1)\beta=x_1$, $\ldots,$ $(x_k)=x_k$, звідки випливає рівність $\alpha=\beta\alpha$. Тоді, очевидно, що $\alpha\mathscr{O\!\!I}\!_n(L)= \beta\alpha\mathscr{O\!\!I}\!_n(L).$

Доведення твердження, що з включення $\mathscr{O\!\!I}\!_n(L)\alpha\subseteq\mathscr{O\!\!I}\!_n(L)\alpha\beta$ для $\alpha,\beta\in\mathscr{O\!\!I}\!_n(L)$ випливає рівність $\mathscr{O\!\!I}\!_n(L)\alpha=\mathscr{O\!\!I}\!_n(L)\alpha\beta$, аналогічне з точністю до дуальності.
\end{proof}

\begin{proposition}\label{proposition-2.2}
\begin{enumerate}
  \item[(1)] Ненульовий елемент $\alpha$ напівгрупи $\mathscr{O\!\!I}\!_n(L)$ є ідемпотентом тоді і лише тоді, коли $\alpha$~--- тотожне часткове перетворення.

  \item[(2)] $\alpha\preccurlyeq\beta$ в $\mathscr{O\!\!I}\!_n(L)$ тоді і лише тоді, коли $\alpha\subseteq \beta$.

  \item[(3)] $\alpha\mathscr{R}\beta$ в $\mathscr{O\!\!I}\!_n(L)$ тоді і лише тоді, коли $\operatorname{dom}\alpha=\operatorname{dom}\beta$.

  \item[(4)] $\alpha\mathscr{L}\beta$ в $\mathscr{O\!\!I}\!_n(L)$ тоді і лише тоді, коли $\operatorname{ran}\alpha=\operatorname{ran}\beta$.

  \item[(5)] $\alpha\mathscr{H}\beta$ в $\mathscr{O\!\!I}\!_n(L)$ тоді і лише тоді, коли $\alpha=\beta$.

  \item[(6)] $\alpha\mathscr{D}\beta$ в $\mathscr{O\!\!I}\!_n(L)$ тоді і лише тоді, коли $|\operatorname{dom}\alpha|=|\operatorname{dom}\beta|$.

  \item[(7)] $\mathscr{D}=\mathscr{J}$ в $\mathscr{O\!\!I}\!_n(L)$.
\end{enumerate}
\end{proposition}

\begin{proof}
Твердження (1) i (2) випливають з означення напівгрупи $\mathscr{O\!\!I}\!_n(L)$ та описання ідемпотентів симетричного інверсного моноїда $\mathscr{I}(L)$ і природного часткового часткового порядку на ньому (див. \cite[підрозділи~1.1 і 3.2]{Lawson-1998}).

Твердження (3) i (4) випливають з означень відношень Ґріна $\mathscr{R}$ і $\mathscr{L}$ на $\mathscr{I}(L)$ і твердження~3.2.11 з \cite{Lawson-1998}.

Твердження (5) випливає з тверджень (3), (4) і леми \ref{lemma-2.1}.

(6) Нехай $\alpha\mathscr{D}\beta$ в $\mathscr{O\!\!I}\!_n(L)$. Тоді існує частковий порядковий ізоморфізм $\gamma\in\mathscr{O\!\!I}\!_n(L)$ такий, що $\alpha\mathscr{R}\gamma$ і $\gamma\mathscr{L}\beta$. З тверджень (3) i (4) випливає, що $\operatorname{dom}\alpha=\operatorname{dom}\gamma$ і $\operatorname{ran}\gamma=\operatorname{ran}\beta$. Позаяк $\beta$ i $\gamma$~--- часткові бієкції, то
\begin{equation*}
  |\operatorname{dom}\alpha|=|\operatorname{dom}\gamma|=|\operatorname{ran}\gamma|=|\operatorname{ran}\beta|=|\operatorname{dom}\beta|.
\end{equation*}

Припустимо, що $|\operatorname{dom}\alpha|=|\operatorname{dom}\beta|$ для деяких $\alpha,\beta\in\mathscr{O\!\!I}\!_n(L)$, і нехай
\begin{equation*}
  \alpha=
\begin{pmatrix}
x_1 & \cdots & x_k\\
y_1 & \cdots & y_k
\end{pmatrix}
\qquad
\hbox{i}
\qquad
\beta=
\begin{pmatrix}
u_1 & \cdots & u_k\\
v_1 & \cdots & v_k
\end{pmatrix},
\end{equation*}
для деяких $x_1,\ldots x_k,y_1,\ldots y_k,u_1,\ldots u_k,v_1,\ldots v_k\in L$, $x_1<\cdots< x_k$, $y_1<\cdots< y_k$, $u_1<\cdots< u_k$, $v_1<\cdots< v_k$, $k\leq n$. Приймемо
\begin{equation*}
    \gamma=
\begin{pmatrix}
y_1 & \cdots & y_k\\
u_1 & \cdots & u_k
\end{pmatrix}.
\end{equation*}
З тверджень (3) i (4) випливає, що $\alpha\mathscr{L}\gamma$ і $\gamma\mathscr{R}\beta$ в $\mathscr{O\!\!I}\!_n(L)$, а отже, $\alpha\mathscr{D}\beta$ в $\mathscr{O\!\!I}\!_n(L)$.

(7) За теоремою 1 з \cite{Koch-Wallace-1957} у стійких напівгрупах відношення Ґріна $\mathscr{D}$ i $\mathscr{J}$ збігаються. Далі ско\-рис\-таємося теоремою~\ref{theorem-2.1}.
\end{proof}

Нагадаємо \cite{Lawson-1998}, що інверсна напівгрупа $S$ називається \emph{комбінаторною}, якщо кожен її $\mathscr{H}$-клас є одноелементною множиною. З твердження \ref{proposition-2.2}(5) випливає

\begin{corollary}
$\mathscr{O\!\!I}\!_n(L)$ --- комбінаторна інверсна напівгрупа.
\end{corollary}

Також з висловлень (6) i (7) твердження \ref{proposition-2.2} випливає такий наслідок.

\begin{corollary}\label{corollary-2.3}
Нехай $n$~--- довільне натуральне число. Тоді сім'я $$\left\{I_k=\mathscr{O\!\!I}\!_k(L)\colon k=0,1,\ldots,n\right\}$$ міс\-тить усі двобічні ідеали напівгрупи $\mathscr{O\!\!I}\!_n(L)$.
\end{corollary}

\begin{lemma}\label{lemma-2.5}
Нехай $\mathfrak{C}$~--- конґруенція на напівгрупі $\mathscr{O\!\!I}\!_n(L)$. Якщо $\alpha\mathfrak{C}\boldsymbol{0}$ для деякого $\alpha\in \mathscr{O\!\!I}\!_n(L)$, то $\alpha\mathfrak{C}\beta$ для всіх $\beta\in \mathscr{O\!\!I}\!_n(L)$ таких, що $\operatorname{rank}\beta\leq\operatorname{rank}\alpha$.
\end{lemma}

\begin{proof}
Розглянемо можливі випадки:
\begin{enumerate}
  \item[(1)] $\operatorname{rank}\beta=\operatorname{rank}\alpha$;
  \item[(2)] $\operatorname{rank}\beta<\operatorname{rank}\alpha$ i $\beta\preccurlyeq\alpha$;
  \item[(3)] $\operatorname{rank}\beta<\operatorname{rank}\alpha$ i $\beta\not\preccurlyeq\alpha$.
\end{enumerate}

(1) Якщо $\operatorname{rank}\beta=\operatorname{rank}\alpha$, то $|\operatorname{dom}\beta|=|\operatorname{ran}\beta|=|\operatorname{ran}\alpha|=|\operatorname{dom}\alpha|$. За лемою~\ref{lemma-2.1} існують часткові порядкові ізоморфізми $\gamma,\delta\in \mathscr{O\!\!I}\!_n(L)$ такі, що $\operatorname{dom}\beta=\operatorname{dom}\gamma$, $\operatorname{ran}\gamma=\operatorname{dom}\alpha$, $\operatorname{dom}\delta=\operatorname{ran}\alpha$ i $\operatorname{ran}\delta=\operatorname{ran}\beta$. Тоді маємо, що $\beta=\gamma\alpha\delta\mathfrak{C}\gamma\boldsymbol{0}\delta=\boldsymbol{0}$, а отже, $\beta\mathfrak{C}\boldsymbol{0}$ і $\beta\mathfrak{C}\alpha$.

(2) Якщо $\operatorname{rank}\beta<\operatorname{rank}\alpha$ i $\beta\preccurlyeq\alpha$, то за лемою~1.4.6 з \cite{Lawson-1998} $\beta=\beta\beta^{-1}\alpha\mathfrak{C}\beta\beta^{-1}\boldsymbol{0}$, а отже, $\beta\mathfrak{C}\boldsymbol{0}$ і $\beta\mathfrak{C}\alpha$.

(3) Припустимо, що $\operatorname{rank}\beta<\operatorname{rank}\alpha$ i $\beta\not\preccurlyeq\alpha$. Нехай $A$ --- підмножина в $\operatorname{dom}\alpha$ така, що $|A|=\operatorname{rank}\beta<\operatorname{rank}\alpha$. Позначимо через $\varepsilon$  тотожне відображення на множині $A$. Тоді $\varepsilon\alpha\preccurlyeq\alpha$ i $\operatorname{rank}(\varepsilon\alpha)=\operatorname{rank}\beta$. З випадку (1) випливає, що $\varepsilon\alpha\mathfrak{C}\alpha\mathfrak{C}\boldsymbol{0}$, і за випадком (2) маємо, що $\varepsilon\alpha\mathfrak{C}\beta$, а отже, $\alpha\mathfrak{C}\beta$.
\end{proof}

\begin{lemma}\label{lemma-2.6}
Нехай $\mathfrak{C}$~--- конґруенція на напівгрупі $\mathscr{O\!\!I}\!_n(L)$ i $\alpha\mathfrak{C}\beta$ для деяких різних $\alpha,\beta\in \mathscr{O\!\!I}\!_n(L)$ таких, що $\beta\preccurlyeq\alpha$. Тоді $\alpha\mathfrak{C}\gamma$ для всіх $\gamma\in \mathscr{O\!\!I}\!_n(L)$ таких, що $\operatorname{rank}\gamma\leq\operatorname{rank}\alpha$.
\end{lemma}

\begin{proof}
Позаяк $\mathscr{O\!\!I}\!_n(L)$~--- інверсна напівгрупа, то з тверджень~2.3.4(1) і~1.4.7(4) монографії \cite{Lawson-1998} випливає, що $\alpha\alpha^{-1}\mathfrak{C}\beta\beta^{-1}$ і $\alpha\alpha^{-1}\preccurlyeq\beta\beta^{-1}$. Також, $\operatorname{rank}\alpha\alpha^{-1}=\operatorname{rank}\alpha$ i $\operatorname{rank}\beta\beta^{-1}=\operatorname{rank}\beta$.
Отже, не зменшуючи загальності, можемо вважати, що $\alpha$ i $\beta$~--- ідемпотенти напівгрупи $\mathscr{O\!\!I}\!_n(L)$.

Припустимо, що $\operatorname{rank}\alpha=k$ для деякого натурального числа $k\leq n$. Якщо $\operatorname{rank}\beta=0$, то скористаємося лемою~\ref{lemma-2.5}. Тому надалі будемо вважати, що $p=\operatorname{rank}\beta\neq 0$ i $p<k$.

Спочатку припустимо, що $p=k-1$ i $\operatorname{dom}\alpha=\left\{x_1,\ldots,x_k\right\}$. Позначимо $\beta_1=\beta$. Позаяк $L$~--- лінійно впорядкована множина, то за лемою~\ref{lemma-2.1} існує частковий порядковий ізоморфізм $\iota_1\in\mathscr{O\!\!I}\!_n(L)$ такий, що $\operatorname{dom}\iota_1=\operatorname{dom}\beta_1$ i $\operatorname{ran}\iota_1=\operatorname{dom}\alpha\setminus\{y_1\}$, де $y_1=\max\operatorname{dom}\beta_1$. Тоді $\iota_1\alpha\iota_1^{-1}=\beta_1$ i $\iota_1\beta_1\iota_1^{-1}\neq\beta_1$. Елемент $\iota_1\beta_1\iota_1^{-1}$ є ідемпотентом, оскільки
\begin{equation*}
\iota_1\beta_1\iota_1^{-1}\iota_1\beta_1\iota_1^{-1}=\iota_1\iota_1^{-1}\iota_1\beta_1\beta_1\iota_1^{-1}=\iota_1\beta_1\iota_1^{-1}.
\end{equation*}
Також, позаяк $\operatorname{dom}\iota_1=\operatorname{dom}\beta_1$, то $\beta_1\iota_1=\iota_1$, звідки випливає, що $\beta_1\iota_1\beta_1\iota_1^{-1}=\iota_1\beta_1\iota_1^{-1}$, а отже, $\iota_1\beta_1\iota_1^{-1}\preccurlyeq\beta$. Оскільки $\alpha\mathfrak{C}\beta=\beta_1$, то $\beta_2=\iota_1\beta_1\iota_1^{-1}\mathfrak{C}\iota_1\alpha\iota_1^{-1}=\beta_1$, а отже, $\beta_2\mathfrak{C}\beta_1\mathfrak{C}\alpha$. З визначення елемента $\beta_2\in\mathscr{O\!\!I}\!_n(L)$ випливає, що
\begin{equation*}
\operatorname{rank}\beta_2=\operatorname{rank}\beta_1-1=k-2.
\end{equation*}

Далі за індукцією для довільного $m=2,\ldots,k$ побудуємо послідовності елементів  $\iota_m\in\mathscr{O\!\!I}\!_n(L)$ та ідемпотентів $\beta_{m+1}\in\mathscr{O\!\!I}\!_n(L)$, які задовольняють  умови
\begin{enumerate}
  \item[(1)] $\iota_m\beta_{m-1}\iota_m^{-1}=\beta_m$ i $\iota_m\beta_m\iota_m^{-1}\neq\beta_m$;
  \item[(2)] $\beta_{m+1}=\iota_m\beta_m\iota_m^{-1}\mathfrak{C}\alpha$;
  \item[(3)] $\beta_{m+1}\preccurlyeq\beta_m$;
  \item[(4)] $\operatorname{rank}\beta_{m+1}=k-m-1$.
\end{enumerate}
Оскільки $L$~--- лінійно впорядкована множина, то за лемою~\ref{lemma-2.1} існує частковий порядковий ізоморфізм $\iota_m\in\mathscr{O\!\!I}\!_n(L)$ такий, що $\operatorname{dom}\iota_m=\operatorname{dom}\beta_m$ i $\operatorname{ran}\iota_m=\operatorname{dom}\beta_{m-1}\setminus\{y_m\}$, де $y_m=\max\operatorname{dom}\beta_m$.
Тоді $\iota_m\beta_{m-1}\iota_m^{-1}=\beta_m$ i $\iota_m\beta_m\iota_m^{-1}\neq\beta_m$. Елемент $\iota_m\beta_m\iota_m^{-1}$ є ідемпотентом, оскільки
\begin{equation*}
\iota_m\beta_m\iota_m^{-1}\iota_m\beta_m\iota_m^{-1}=\iota_m\iota_m^{-1}\iota_m\beta_m\beta_m\iota_m^{-1}=\iota_m\beta_m\iota_m^{-1}.
\end{equation*}
Також, позаяк $\operatorname{dom}\iota_m=\operatorname{dom}\beta_m$, то $\beta_m\iota_m=\iota_m$, звідки випливає, що $\beta_m\iota_m\beta_m\iota_m^{-1}=\iota_m\beta_m\iota_m^{-1}$, а отже, $\iota_m\beta_m\iota_m^{-1}\preccurlyeq\beta_m$. Оскільки $\alpha\mathfrak{C}\cdots\mathfrak{C}\beta_{m-1}\mathfrak{C}\beta_{m}$, то
\begin{equation*}
\beta_{m+1}=\iota_m\beta_m\iota_m^{-1}\mathfrak{C}\iota_m\beta_{m-1}\iota_m^{-1}=\beta_m,
\end{equation*}
а отже, $\beta_{m+1}\mathfrak{C}\beta_m\mathfrak{C}\cdots\mathfrak{C}\alpha$. З визначення елемента $\beta_{m+1}\in\mathscr{O\!\!I}\!_n(L)$ випливає, що
\begin{equation*}
\operatorname{rank}\beta_{m+1}=\operatorname{rank}\beta_m-1=k-m-1.
\end{equation*}

За вище викладеною побудовою маємо, що $\operatorname{rank}\beta_{k}=0$, а отже, $\beta_{k}=\boldsymbol{0}$~--- нуль напівгрупи $\mathscr{O\!\!I}\!_n(L)$. Отже, отримали, що $\alpha\mathfrak{C}\boldsymbol{0}$. Далі скористаємося лемою~\ref{lemma-2.5}.

Припустимо, що $p<k-1$. Тоді існує елемент $x\in L$ такий, що $x\in\operatorname{dom}\alpha\setminus\operatorname{dom}\beta$. Нехай $\varepsilon$~--- тотожне перетворення множини $\operatorname{dom}\alpha\setminus\{x\}$. Очевидно, що $\beta\preccurlyeq\varepsilon\preccurlyeq\alpha$ i $\operatorname{rank}\varepsilon=\operatorname{rank}\alpha-1$. Тоді $\varepsilon=\varepsilon\alpha\mathfrak{C}\varepsilon\beta=\beta$, а отже, $\varepsilon\mathfrak{C}\beta$. Звідки випливає, що $\varepsilon\mathfrak{C}\alpha$. Далі скористаємося попередньо доведеним фактом.
\end{proof}

\begin{lemma}\label{lemma-2.7}
Нехай $\mathfrak{C}$~--- конґруенція на напівгрупі $\mathscr{O\!\!I}\!_n(L)$ i $\alpha\mathfrak{C}\beta$ для деяких різних $\alpha,\beta\in \mathscr{O\!\!I}\!_n(L)$. Тоді $\alpha\mathfrak{C}\gamma$ для всіх $\gamma\in \mathscr{O\!\!I}\!_n(L)$ таких, що $\operatorname{rank}\gamma\leq\max\{\operatorname{rank}\alpha,\operatorname{rank}\beta\}$.
\end{lemma}

\begin{proof}
Позаяк $\mathscr{O\!\!I}\!_n(L)$~--- інверсна напівгрупа, то за твердженнями~2.3.4(1) і~1.4.7(4) з \cite{Lawson-1998} маємо, що $\alpha\alpha^{-1}\mathfrak{C}\beta\beta^{-1}$ і $\alpha\alpha^{-1}\preccurlyeq\beta\beta^{-1}$. Також, $\operatorname{rank}\alpha\alpha^{-1}=\operatorname{rank}\alpha$ i $\operatorname{rank}\beta\beta^{-1}=\operatorname{rank}\beta$.
Отже, не зменшуючи загальності, можемо вважати, що $\alpha$ i $\beta$~--- ідемпотенти напівгрупи $\mathscr{O\!\!I}\!_n(L)$.

Позаяк $\alpha=\alpha\alpha\mathfrak{C}\alpha\beta\mathfrak{C}\beta\beta=\beta$, то $\alpha\mathfrak{C}\alpha\beta\mathfrak{C}\beta$. З умови $\alpha\neq\beta$ випливає, що $\alpha\beta\preccurlyeq\alpha$ i $\alpha\beta\preccurlyeq\beta$. Далі скористаємося лемою~\ref{lemma-2.6}.
\end{proof}

\begin{theorem}\label{theorem-2.8}
Нехай $(L,\leqslant)$~--- нескінченна лінійно впорядкована множина. Для довільного натурального числа $n$ кожна конґруенція на напівгрупі $\mathscr{O\!\!I}\!_n(L)$ є конґруенцією Ріса.
\end{theorem}

\begin{proof}
Нехай $\alpha$~--- елемент напівгрупи $\mathscr{O\!\!I}\!_n(L)$ з $\operatorname{rank}\alpha=k\leq n$ такий, що виконуються умови:
\begin{enumerate}
  \item[(1)] існує елемент $\beta\neq\alpha$ напівгрупи $\mathscr{O\!\!I}\!_n(L)$ з $\operatorname{rank}\beta\leq \operatorname{rank}\alpha$ такий, що $\alpha\mathfrak{C}\beta$;

  \item[(2)] для довільного $\gamma\in\mathscr{O\!\!I}\!_n(L)$ з $\operatorname{rank}\gamma>k$ елемент $\gamma$ $\mathfrak{C}$-еквівалентний лише $\gamma$ у випадку, коли $k<n$.
\end{enumerate}
За лемою~\ref{lemma-2.7} усі елементи ідеала $I_k=\mathscr{O\!\!I}\!_k(L)$ напівгрупи $\mathscr{O\!\!I}\!_n(L)$ є $\mathfrak{C}$-еквівалентними, а отже, конґруенція $\mathfrak{C}$ породжена ідеалом $I_k$. Очевидно, що одинична та універсальна конґруєції є конґруенціями Ріса на $\mathscr{O\!\!I}\!_n(L)$, оскільки вони породжені ідеалами $I_0=\boldsymbol{0}$ i $I_n=\mathscr{O\!\!I}\!_n(L)$, відповідно.
\end{proof}

\begin{remark}
У праці \cite{Gutik-Popadiuk=2023} доведено, що на напівгрупі $\mathscr{I}_\omega^{n}(\overrightarrow{\mathrm{conv}})$ усіх часткових по\-ряд\-ко\-во-опуклих ізоморфізмв лінійно впорядкованої множини $(\omega,\leq)$ рангу $\leqslant n$ відношення Ґріна $\mathscr{D}$ i $\mathscr{J}$ збігаються, і крім того, кожна конґруенція на $\mathscr{I}_\omega^{n}(\overrightarrow{\mathrm{conv}})$ є конґруенцією Ріса. Очевидно, що для довільного натурального числа $n$ напівгрупа $\mathscr{I}_\omega^{n}(\overrightarrow{\mathrm{conv}})$ є піднапівгрупою в $\mathscr{O\!\!I}\!_n(\omega)$.
\end{remark}

Нагадаємо \cite{Gutik-Lawson-Repov=2009}, що нескінченна підмножина $D$ напівгрупи $S$ називається \emph{$\omega$-нестій\-кою}, якщо $sB\cup Bs\nsubseteq D$ для довільних $s\in D$ і нескінченої підмножини $B$ у $D$. Будемо говорити, що \emph{напівгрупа $S$ має щільний ряд ідеалів} $J_0\subseteq J_1\subseteq\cdots\subseteq J_m$, $m\in\mathbb{N}$, якщо $J_0$~--- скінченний ідеал в $S$ i $J_k\setminus J_{k-1}$~--- $\omega$-нестійка підмножина в $S$ для довільного $k=1,\ldots,m$.

Надалі через $I_0\subseteq I_1\subseteq \cdots\subseteq I_n$ позначатимемо ряд ілеалів, визначений у формулюванні наслід\-ку~\ref{corollary-2.3}. Очевидно, що $I_0=\{\boldsymbol{0}\}$. Також з твердження~\ref{proposition-2.2} і наслідку~\ref{corollary-2.3} випливає, що нескінченна множина $I_k\setminus I_{k-1}$ є $\mathscr{D}$-класом напівгрупи $\mathscr{O\!\!I}\!_n(L)$, що складається з елементів рангу $k$. Отже, для довільного елемента $\alpha\in I_k\setminus I_{k-1}$ і довільної нескінченної множини $B\subseteq I_k\setminus I_{k-1}$ існує елемент $\beta\in I_k\setminus I_{k-1}$ такий, що $\operatorname{dom}\beta\neq\operatorname{ran}\alpha$ або $\operatorname{dom}\alpha\neq\operatorname{ran}\beta$. З вище наведених міркувань і визначення напівгрупової операції на $\mathscr{O\!\!I}\!_n(L)$ випливає, що $\left\{\alpha\beta,\beta\alpha\right\}\nsubseteq I_k\setminus I_{k-1}$, а отже, $I_k\setminus I_{k-1}$~--- $\omega$-нестійка підмножина в напівгрупі $\mathscr{O\!\!I}\!_n(L)$ для довільного $k=1,\ldots,n$. Отож, ми довели таке твердження

\begin{proposition}\label{proposition-2.4}
$I_0\subseteq I_1\subseteq \cdots\subseteq I_n$~--- щільний ряд ідеалів у $\mathscr{O\!\!I}\!_n(L)$ для довільного натурального числа $n$.
\end{proposition}

Відомо, що напівгрупи зі щільними рядами ідеалів зберігаються скінченними прямими до\-бут\-ка\-ми та гомоморфізмами, для яких кожен елемент з образу має скінченний про\-образ~ \cite{Gutik-Lawson-Repov=2009}.

\begin{theorem}\label{theorem-2.9}
Нехай $(L,\leqslant)$~--- нескінченна лінійно впорядкована множина, $n$ --- довільне натуральне число, $S$~--- напівгрупа та $\mathfrak{h}\colon \mathscr{O\!\!I}\!_n(L)\to S$~--- неанулюючий гомоморфізм. Тоді гомоморфний образ $(\mathscr{O\!\!I}\!_n(L))\mathfrak{h}$~--- напівгрупа зі щільними рядами ідеалів.
\end{theorem}

\begin{proof}
Твердження теореми очевидне у випадку, коли $\mathfrak{h}\colon \mathscr{O\!\!I}\!_n(L)\to S$~--- ін'єк\-тив\-ний гомоморфізм. Позаяк $\mathfrak{h}\colon \mathscr{O\!\!I}\!_n(L)\to S$~--- неанулюючий гомоморфізм і за теоремою \ref{theorem-2.8} кожна конґруенція на напівгрупі є конґруенцією Ріса, що існує натуральне число $k\leq n$ таке, то виконуються такі умови:
\begin{enumerate}
  \item[(1)] $(\alpha)\mathfrak{h}=(\beta)\mathfrak{h}$ для довільних $\alpha,\beta\in I_k$;
  \item[(2)] $(\gamma)\mathfrak{h}\neq(\delta)\mathfrak{h}$ для довільних різних $\gamma,\delta\in \mathscr{O\!\!I}\!_n(L)\setminus I_k$.
\end{enumerate}
Отож, отримуємо, що гомоморфний образ $(I_k)\mathfrak{h}$~--- нуль напівгрупи $(\mathscr{O\!\!I}\!_n(L))\mathfrak{h}$, а також, що звуження $\mathfrak{h}{\upharpoonright}_{\mathscr{O\!\!I}\!_n(L)\setminus I_k}\colon \mathscr{O\!\!I}\!_n(L)\setminus I_k\to S$ гомоморфізму $\mathfrak{h}$ є ін'єктивним відображенням. З останньої умови випливає, що множина $(I_m)\mathfrak{h}\setminus (I_{m-1})\mathfrak{h}$~--- $\omega$-нестійка підмножина в напівгрупі $(\mathscr{O\!\!I}\!_n(L))\mathfrak{h}$ для довільного $m=k+1,\ldots,n$. Справді, для довільних елементів $a,b\in(I_m)\mathfrak{h}\setminus (I_{m-1})\mathfrak{h}$ ($m=k+1,\ldots,n$), їхні повні прообрази $(a)\mathfrak{h}^{-1}$ i $(b)\mathfrak{h}^{-1}$ є одно\-еле\-мент\-ними підмножинами в напівгрупі $\mathscr{O\!\!I}\!_n(L)$. Нехай $\alpha=(a)\mathfrak{h}^{-1}$ i $\beta=(b)\mathfrak{h}^{-1}$. Тоді $\alpha,\beta\in I_m\setminus I_{m-1}$, і з визначення напівгрупової операції на $\mathscr{O\!\!I}\!_n(L)$ випливає, що $\left\{\alpha\beta,\beta\alpha\right\}\nsubseteq I_m\setminus I_{m-1}$. Звідси, врахувавши умови (1) і (2), отримуємо, що виконується хоча б одна з умов або
\begin{equation*}
ab=(\alpha)\mathfrak{h}(\beta)\mathfrak{h}=(\alpha\beta)\mathfrak{h}\notin (I_m)\mathfrak{h}\setminus (I_{m-1})\mathfrak{h},
\end{equation*}
або
\begin{equation*}
b a=(\beta)\mathfrak{h}(\alpha)\mathfrak{h}=(\beta\alpha)\mathfrak{h}\notin (I_m)\mathfrak{h}\setminus (I_{m-1})\mathfrak{h},
\end{equation*}
тобто $(I_m)\mathfrak{h}\setminus (I_{m-1})\mathfrak{h}$~--- $\omega$-нестійка підмножина в $(\mathscr{O\!\!I}\!_n(L))\mathfrak{h}$ для $m=k+1,\ldots,n$. Отже, виконується твердження теореми.
\end{proof}


\section*{\textbf{Подяка}}

Автори висловлюють щиру подяку  рецензентові за цінні поради та зауваження.


\end{document}